\documentclass{coulonpaper}
\setbftrue
\mcgfrenchtrue

\title{Examples of groups whose automorphisms have exotic growth}
\author{R\'emi Coulon}

\begin{document}


\maketitle

\begin{abstract}
	In this note we produce examples of outer automorphisms of finitely generated groups which have exotic behaviors in terms of growth of conjugacy classes. 
\end{abstract}

%
\section{Introduction}
%

\paragraph{Growth of outer automorphisms.}
Let $G$ be a finitely generated group.
The outer automorphism group of $G$, which we denote by $\out G$, naturally acts on the set of conjugacy classes of $G$.
It is a standard strategy to study the properties of an automorphism $\Phi \in \out G$ by looking at the dynamics of this action.
Here is an example. 
Endow $G$ with the word metric. 
Given any conjugacy class $c$ of $G$, denote by $\norm c$ the length of the smallest element in $c$.
One may wonder what is the \emph{growth of $\Phi$}, i.e. the asymptotic behavior of the map $\N \to \R$ sending $n$ to $\norm{\Phi^n(c)}$.
It turns out that for various intensively studied groups, one observes a strong \emph{growth dichotomy}.

\begin{theo}
\label{res: growth dichotomy}
	Let $G$ be a finitely generated group in one of the following classes:
	free abelian groups, free groups, fundamental groups of closed compact orientable surfaces, torsion-free hyperbolic groups, torsion-free toric relatively hyperbolic groups.
	Then for every $\Phi \in \out G$, for every conjugacy class $c$ of $G$, the map 
	\begin{equation*}
		\begin{array}{ccc}
		 	\N & \to & \R \\
			n & \to & \norm{\Phi^n(c)}
		\end{array}
	\end{equation*}
	grows either polynomially or at least exponentially.
\end{theo}

In the above statement, the growth type of a map is understood in the following sense: given two functions $f_1, f_2 \colon \N \to \R$, we say that $f_1$ \emph{grows at most like} $f_2$ (or $f_2$ \emph{grows at least like} $f_1$) and we write $f_1 \prec f_2$, if there exists $C > 0$ such that for every $n \in \N$,
\begin{equation*}
	f_1(n) \leq C f_2(n) + C.
\end{equation*}
We say that $f_1$ \emph{grows like} $f_2$ (or $f_1$ and $f_2$ are \emph{equivalent}) and we write $f_1 \asymp f_2$, if $f_1 \prec f_2$ and $f_2 \prec f_1$.
A map $f \colon \N \to \R$ \emph{grows polynomially} (\resp \emph{exponentially}) if it grows like $n \to n^d$, for some exponent $d \in \N$, (\resp $n \to \lambda^n$, for some $\lambda > 1$).

If $G$ is a free abelian group, then $\Phi$ can be seen as a matrix $M$ in $\operatorname{GL}(n, \Z)$ and \autoref{res: growth dichotomy} follows from the Jordan decomposition of $M$.
When $G$ is a surface group, the statement is a consequence of the Nielsen-Thurston classification of mapping classes; see for instance \cite[Theorem~13.2]{Farb:2012ws}.
For non-abelian free groups, a complete classification of the possible growth types of $n \to \norm{\Phi^n(c)}$ has been given by Levitt \cite{Levitt:2009hx}.
It builds on the theory of train-track representatives developed by Bestvina and Handel \cite{Bestvina:1992wa}.
The case of torsion-free hyperbolic and toric relatively hyperbolic groups is handled in \cite{Coulon:2019aa}.

In all these examples the growth type of $\Phi$ directly relates to the algebraic/geometric properties of $\Phi$.
Let us mention one example.
Let $\Sigma$ be a closed compact surface of genus at least $2$ and $G$ its fundamental group.
Let $f \in \mcg\Sigma$ be an element of the mapping class group and $\Phi$ the corresponding outer automorphism of $G$.
Let $\alpha$ be a closed curve on $\Sigma$ and $c$ the associated conjugacy class of $G$.
Roughly speaking, if $n \to\norm{\Phi^n(c)}$ grows exponentially then $\alpha$ crosses an $f$-invariant subsurface on which $f$ behaves as a pseudo-Anosov element.
In addition the exponential growth rate of $n \to\norm{\Phi^n(c)}$ directly relates to the stretching factor of the underlying pseudo-Anosov homeomorphism.
On the other hand, if $n \to \norm{\Phi^n(c)}$ is linear then $c$ lies on a subsurface on which $f$ is equal to a product of Dehn twists along pairwise disjoint curves.

\paragraph{Main results.}
This article goes in the opposite direction.
Our goal is to produce examples of groups where such growth dichotomy fails.
More generally, we try to understand what are the possible growth types for the map $n \to \norm{\Phi^n(c)}$.
In order to state our main result, we first recall the definition of a length function.

\begin{defi}
\label{def: length function}
	A \emph{length function} on a group $H$ is a map $L \colon H \to \N$ with the following properties
    \begin{enumerate}
    	\item $L(h) = 0$ if and only if $h = 1$;
    	\item $L(h) = L(h^{-1})$, for every $h \in H$;
    	\item $L(h_1h_2) \leq L(h_1) + L(h_2)$, for every $h_1,h_2 \in H$;
    	\item there exists $\lambda > 0$ such that for every $r \in \N$, the set $\set{h \in H}{L(h) \leq r}$ contains at most $\lambda^r$ elements.
    \end{enumerate}
\end{defi}

For instance, the map $n \to \abs n ^\alpha$ is a length function on $\Z$, for every $\alpha \in (0,1)$.

\begin{theo}
\label{res: growth types}
	Let $\mathcal L$ be a finite collection of computable length functions on $\Z$.
	There exist a finitely generated group $G$ and an automorphism $\Phi \in \out G$ with the following property.
	For every non-trivial conjugacy class $c$ of $G$, the map
	\begin{equation*}
		\begin{array}{lccc}
			T_c \colon &\Z & \to & \R \\
			& n & \to & \ln \norm{\Phi^n(c)}
		\end{array}
	\end{equation*}
	is equivalent to a linear combination of the elements of $\mathcal L$.
	Conversely for every $L \in \mathcal L$ there exits a conjugacy class $c$ of $G$ such that $T_c$ is equivalent to $L$.
\end{theo}

Recall that there are infinitely many inequivalent computable length functions (actually countably many).
Hence \autoref{res: growth types} provides numerous examples of exotic automorphisms.
As we explain in the last section, our result is more than just a theoretical existence statement.
In many cases, one can build an explicit presentation of such a group $G$.

Unlike in \autoref{res: growth dichotomy}, our techniques are not accurate enough to get a precise behavior of the map $n \to \norm{\Phi^n(c)}$. 
We only control its logarithm.
For instance the last part of \autoref{res: growth types} can be reformulated as follows: for every $L \in \mathcal L$, there exist a conjugacy class $c$ of $G$, $\lambda_1, \lambda_2 > 1$ and $A > 0$ such that for every $n \in \N$, we have
\begin{equation*}
	\frac 1A \lambda_1^{L(n)} \leq \norm{\Phi^n(c)} \leq A \lambda_2^{L(n)}.
\end{equation*}
Nevertheless, as we explain below, the map $n \to \ln \norm{\Phi^n(c)}$ has a natural interpretation in terms of the Lipschitz distance on $\out G$.

\medskip
In the second part of our work, we stop focusing on a single outer automorphism and adopt a more global point of view.
The word metric on a finitely generated group $G$ induces an asymmetric left-invariant pseudo-metric on $\out G$, called the \emph{Lipschitz metric}, and defined as follows:
for every $\Phi_1, \Phi_2 \in \out G$,
\begin{equation*}
	\dist[\rm Lip]{\Phi_1}{\Phi_2} = \ln \left( \sup_c \frac{\norm{\Phi_2^{-1} (c)}}{\norm {\Phi_1^{-1}(c)}} \right)
\end{equation*}
where $c$ runs over all non-trivial conjugacy classes of $G$.
This distance is directly inspired by the Thurston metric on the Teichmüller space \cite{Thurston:1998aa} and the Lipschitz metric on the Culler-Vogtmann outer space \cite{Francaviglia:2011jj, Meinert:2015ce}.

\begin{rema*}\
	\begin{itemize}
		\item Let $S$ and $S'$ be two generating sets of $G$.
		It is well-known that the corresponding word metrics on $G$ are bi-Lipschitz equivalent.
		Thus the associated Lipschitz distances on $\out G$ are quasi-isometric.
		As we work up to quasi-isometry, the choice of a generating set for $G$ does not really matter.

		\item 
		Note that if $\out G$ is finitely generated, the Lipschitz metric $\distV[\rm Lip]$ is a priori not quasi-isometric to the word metric $\distV$. 
		Consider for instance the free group $G = \F(a,b)$.
		Let $\phi$ be the automorphism characterized by 
		\begin{equation*}
			\begin{array}{lrrl}
				\phi \colon & \F(a,b) & \to & \F(a,b) \\
				& a & \to & a \\
				& b & \to & ba 
			\end{array}
		\end{equation*}
		and $\Phi$ its outer class.
		Note that $\Phi$ can be interpreted as a Dehn twist on a puncture torus.
		It follows that $n \to \dist[\rm Lip] 1{\Phi^n}$ grows logarithmically.
		However $\out G = {\rm GL}_2(\Z)$ is hyperbolic, thus $n \to \dist 1{\Phi^n}$ grows linearly.
		Hence $\distV[\rm Lip]$ and $\distV$ are not quasi-isometric.
	\end{itemize}
\end{rema*}

It is known that any finitely presented group $Q$ is isomorphic to $\out G$ for some suitable finitely generated group $G$; see for instance Matumoto \cite{Matumoto:1989wg}.
The next statement can be seen as a geometric analogue of this realization result.

\begin{theo}
\label{res: qi}
	For every finitely presented group $Q$, there exists a finitely generated group $G$ such that $\out G$ endowed with the Lipschitz metric is quasi-isometric to $Q$.
\end{theo}

\paragraph{Strategy.}
Our work relies on the Rips construction \cite{Rips:1982co}. 
If $Q$ is a finitely presented group, Rips explains how to build a non-elementary hyperbolic group $G$ such that $Q$ is the quotient of $G$ by a \emph{finitely generated} normal subgroup $N$.
The action of $G$ on $N$ by conjugation defines a map $G \to \aut N$ which induces a homomorphism $\chi \colon Q \to \out N$.
It turns out that this map is a quasi-isometric embedding, provided $N$ is not finite.
The proof has two steps.
\begin{enumerate}
	\item Classical arguments show that the word metric on $Q$ \og dominates \fg the Lipschitz metric on $\out N$ (\autoref{res upper bound}).
	\item Building on the fact that the divergence function in a hyperbolic space is exponential, one provides an estimate from below of the Lipschitz metric (\autoref{res: lower bound}).
\end{enumerate}
Once the relation between the word metric on $Q$ and the Lipschitz metric on $\out N$ is established, we vary the group $Q$ to get the above theorems:
\begin{itemize}
	\item According to Ol'shanski\u\i\ \cite{Olshanskii:1997ga}, for any computable length function $L$ on $\Z$, there exist a finitely presented group $Q$ and an element $q$ such that $L$ is equivalent to the word metric of $Q$ restricted to $\group q$.
	If $\group q$ is distorted in $Q$, then the corresponding automorphism $\Phi = \chi(q)$ has an exotic growth type, which yields \autoref{res: growth types}.
	\item Bumagin and Wise proved that the Rips construction can be refined so that the map $\chi \colon Q \to \out N$ is actually an isomorphism \cite{Bumagin:2005fr}, leading to \autoref{res: qi}.
\end{itemize}

\paragraph{Acknowledgment.}
The author is grateful to the \emph{Centre Henri Lebesgue} ANR-11-LABX-0020-01 for creating an attractive mathematical environment.
He acknowledges support from the Agence Nationale de la Recherche under Grant \emph{Dagger} ANR-16-CE40-0006-01.

He warmly thanks Gilbert Levitt, Arnaud Hilion, and Camille Horbez for cheerful and stimulating discussions around this work.

%
\section{Exploiting a short exact sequence}
%
\label{sec: short exact sequence}

\begin{nota*}
	Let $G$ be a group and $\distV$ a left-invariant (pseudo-)distance on $G$.
	In order to lighten the notations, for every $g \in G$, we write $\abs g = \dist 1g$ for the \emph{length} of $g$.
	It induces a \emph{norm} on the set of conjugacy classes of $G$: if $c$ is a conjugacy class, $\norm c$ is the length of the shortest element in $c$.
	Given $g \in G$, we make an abuse of notation and write also $\norm g$ for the norm of the conjugacy class of $g$.
	
	In the course of the article, we will work with several metrics. 
	If one needs to avoid ambiguity, we will use indices of the form $\distV[*]$ to distinguish between them.
	The associated length function and norm will be then denoted by $\absV[*]$ and $\normV[*]$ respectively.
\end{nota*}

We now fix once for all a short exact sequence of groups
\begin{equation*}
	1 \to N \to G \to Q \to 1
\end{equation*}
where $N$ and $Q$ are finitely generated, whereas $G$ is non-elementary hyperbolic in the sense of Gromov \cite{Gromov:1987tk}.
We endow $N$ and $Q$ with the word metric.
The distance on $N$ induces a norm on the set of conjugacy classes in $N$, denoted by $\normV$.
The action by conjugation of $G$ on $N$ defines a map $G \to \aut N$, which induces a morphism $\chi \colon Q \to \out N$.
The goal of this section is to prove the following statement.

\begin{prop}
\label{res: qi embedding}
	If $N$ is infinite, then the map $\chi \colon Q \to \out N$ is a quasi-isometric embedding.
	More precisely, there exist $\kappa \geq 1$ and $\ell \geq 0$ such that for every $q,q' \in Q$,
	\begin{equation*}
		\kappa^{-1} \dist {q'}q - \ell \leq \dist[\rm Lip]{\chi(q')}{\chi(q)} \leq \kappa \dist {q'}q.
	\end{equation*}
\end{prop}

We start with a very general fact which does not require any kind of negative curvature.
\begin{prop}
\label{res upper bound}
	There exists $\kappa > 0$ such that for every $q,q' \in Q$, 
	\begin{equation}
	\label{eqn: upper bound}
		\dist[\rm Lip]{\chi(q')}{\chi(q)} \leq \kappa \dist {q'}q.
	\end{equation}
\end{prop}

\begin{proof}
	Since both metrics on $Q$ and $\out N$ are left-invariant, it suffices to prove the inequality when $q' = 1$.
	Let $S$ be the generating set of $Q$ defining the word metric.
	We let 
	\begin{equation*}
		\kappa = \max_{s \in S \cup S^{-1}} \abs[\rm Lip]{\chi(s)}.
	\end{equation*}
	Let $q \in Q$.
	We decompose $q$ as a geodesic word $q = s_1s_2\dots s_m$ over the alphabet $S \cup S^{-1}$, i.e. $m = \abs q$.
	It follows from the triangle inequality that
	\begin{equation*}
		\abs[\rm Lip] {\chi(q)} 
		\leq \sum_{k=1}^m\dist[\rm Lip]{\chi(s_1\dots s_{k-1})}{\chi(s_1\dots s_k)}
		\leq \sum_{k=1}^m\abs[\rm Lip]{\chi(s_k)}
		\leq \kappa m. \qedhere
	\end{equation*}
\end{proof}

In order to bound from below the Lipschitz metric we start by studying the behavior of a single conjugacy class under the action of $\out G$.

\begin{prop}
\label{res: lower bound}
	There exists $\kappa > 0$ with the following property.
	Let $c$ be a conjugacy class of $N$, whose elements have infinite order.
	Then there exists $\ell \geq 0$, such that for all $q \in Q$, we have
	\begin{equation*}
		\ln \norm {\Phi(c)} \geq \kappa \abs q - \ell,
		\quad \text{where} \quad
		\Phi = \chi(q).
	\end{equation*}
\end{prop}

\begin{proof}
	We start by setting up all the useful objects for this proof.
	Let $S_N$ and $S_Q$ be the respective generating sets of $N$ and $Q$ used to define the word metric $\distV[N]$ and $\distV[Q]$ on these groups.
	We fix a generating set $S_G$ of $G$ which contains $S_N$ and whose image in $Q$ lies in $S_Q$, and endow $G$ with the corresponding word metric $\distV[G]$ (or simply $\distV$).
	It follows that the maps $N \into G$ and $G \onto Q$ are $1$-Lipschitz.
		
	Let $X$ (\resp $Y$) be the Cayley graph of $G$ (\resp $N$) with respect to $S_G$ (\resp $S_N$).
	We identify the vertex set of $X$ with $G$.
	Although it may not be unique, given two points $x,y \in X$, we write $\geo xy$ for some geodesic joining them.
	Since $S_N$ is contained in $S_G$, there exists a (unique) $1$-Lipschitz $N$-equivariant graph embedding $\iota \colon Y \into X$ sending the identity of $N$ to the one of $G$.
	For simplicity we don't make any distinction (as a set) between $Y$ and its image in $X$.

	By assumption the group $G$ is hyperbolic, i.e. there exists $\delta \in \R_+$, such that  every geodesic triangle in $X$ is $\delta$-thin, or equivalently for every $x,y,z \in X$, the geodesic $\geo xz$ lies in the $\delta$-neighborhood of $\geo xy \cup \geo yz$.
	Without loss of generality we can assume that $\delta>1$.
	We refer the reader to Gromov's seminal paper \cite{Gromov:1987tk} or \cite{Coornaert:1990tj, Ghys:1990ki, Bridson:1999ky} for the background on hyperbolic geometry.
	The next statement is a straightforward reformulation of \cite[Chapter III.H, Proposition~1.6]{Bridson:1999ky} expressing that the divergence function in a hyperbolic space is exponential.

	\begin{lemm}
	\label{res: div}
		There exist $A > 0$ and $\lambda > 1$ with the following property.
		Let $x,y \in X$ and $p \in \geo xy$.
		Let $r \geq 0$.
		Let $\gamma$ be a rectifiable path from $x$ to $y$.
		If $\gamma$ does not intersect the ball $B(p,r)$, then its length is at least $A \lambda^r$.
	\end{lemm}
		
	Let $c$ be a conjugacy class of $N$ whose elements have infinite order.
	We choose an element $u \in N$ in $c$ with minimal length in $N$, i.e. $\abs[N] u = \norm[N]c$.
	It follows from the Morse Lemma, that the set
	\begin{equation*}
		A(u) = \set{x \in X}{\dist x{ux} \leq \abs[G] u + 4\delta},
	\end{equation*}
	is quasi-isometric to a line, on which the group $\group u$ roughly acts by translations.
	Compare for instance with \cite[Chapitre~10, §~7]{Coornaert:1990tj}.
	In particular, there exists $a \geq 0$, such that the diameter of $A(u) / \group u$ is less than $a$.
	Without loss of generality, we can assume that $a > \abs[G]u + 2\delta$.

	Let $q \in Q$ and $\Phi = \chi(q)$.
	We fix $v \in N$ in the conjugacy class $\Phi(c)$ such that $\abs[N]v = \norm[N]{\Phi(c)}$.
	It follows from the definition of $\chi \colon Q \to \out N$, that there exists a pre-image $h \in G$ of $q$ such that $v = h u h^{-1}$.
	We fix such an element $h$ which minimizes $\abs[G]h$.

	\begin{lemm}
	\label{res: gamma(t) close to z/uz}
		There exists a point $x \in \geo{uh^{-1}}{h^{-1}}$ such that $\dist 1x \leq a + 2\delta$.
	\end{lemm}

	\begin{proof}
		If $\dist 1{h^{-1}} \leq a$, one simply takes $x = h^{-1}$.
		Suppose that $\dist 1{h^{-1}} > a$.
		We write $z$ for the point on a geodesic $\geo 1{h^{-1}}$ such that $\dist 1z = a$ (see \autoref{fig: main fig}).
		\begin{figure}[htbp]
		\begin{center}
			\includegraphics[width=\textwidth]{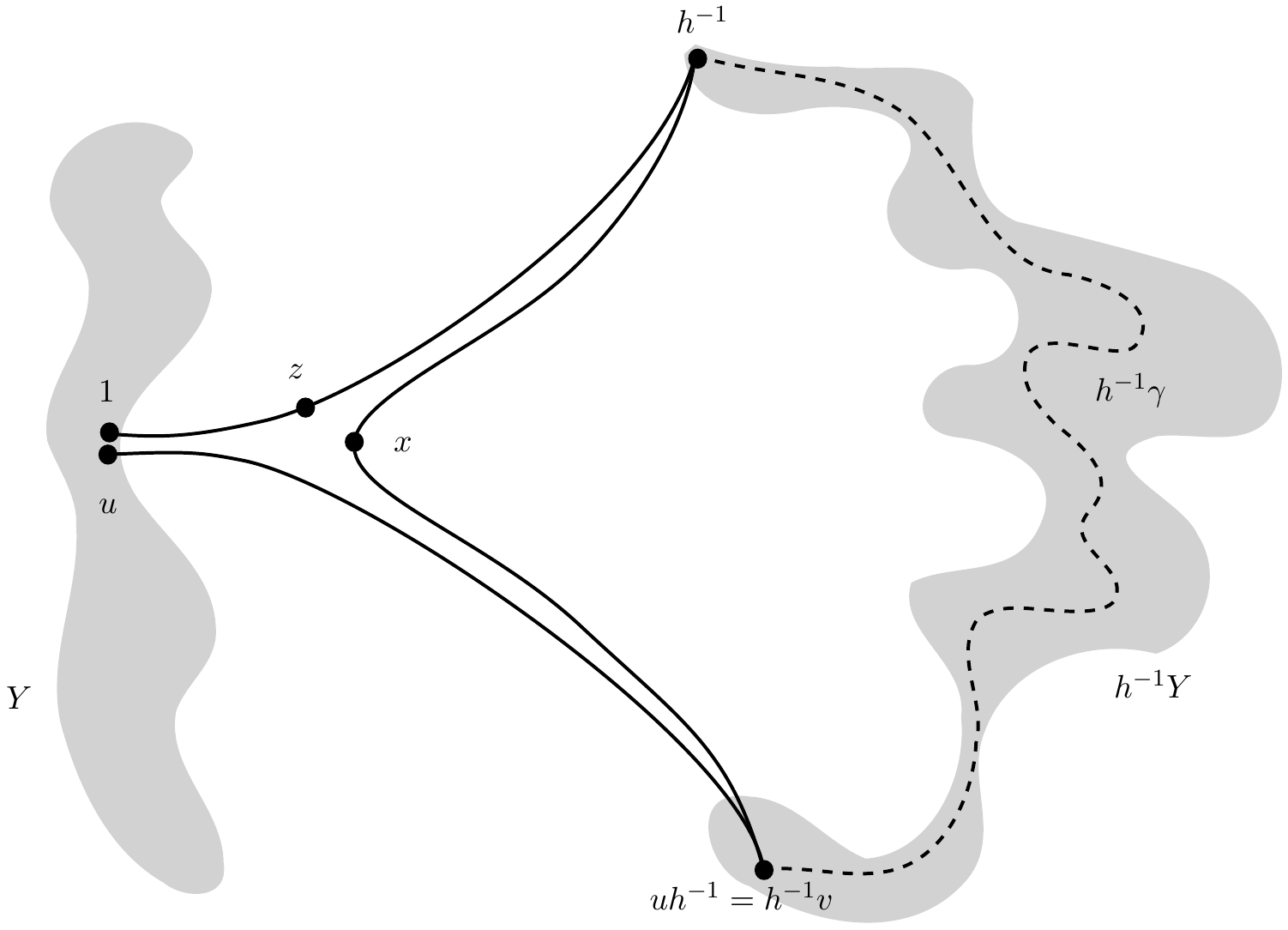}
		\caption{The geodesics $\geo 1{h^{-1}}$, $\geo u{uh^{-1}}$ and $\geo{h^{-1}}{uh^{-1}}$ are the solid lines whereas $h^{-1}\gamma$ is the dashed path.}
		\label{fig: main fig}
		\end{center}
		\end{figure}
		It suffices to show that $z$ is $2\delta$-close to $\geo{uh^{-1}}{h^{-1}}$.
		Recall that geodesic triangles in $X$ are $\delta$-thin.
		Hence $z$ is $2\delta$-close to one of the following geodesics $\geo 1u$, $\geo u{uh^{-1}}$, or $\geo{uh^{-1}}{h^{-1}}$.
		Let us rule out the first two cases.
		Assume that $z$ is $2\delta$-close to $\geo 1u$.
		It follows from the triangle inequality that $a \leq \dist 1z \leq \abs u + 2\delta$,
		which contradicts our choice of $a$.
		On the other hand, if $z$ is $2\delta$-close to a point $y$ on $\geo u{uh^{-1}}$, the triangle inequality yields
		\begin{align*}
			\dist{uz}z
			\leq \abs{\dist u{uz} - \dist uy} + 2\delta
			& \leq \abs{\dist 1z - \dist uy} + 2\delta \\
			& \leq \dist 1u + \dist zy + 2\delta.
		\end{align*}
		Consequently $\dist {uz}z \leq \abs u + 4\delta$, i.e. $z$ belongs to $A(u)$.
		According to our choice of $a$, there exists $m \in \Z$, such that $\dist 1{u^m z} < a \leq \dist 1z$.
		We let $h_0 = hu^{-m}$.
		Note that $h$ and $h_0$ have the same image in $Q$ whereas $h_0 uh_0^{-1} = huh^{-1} = v$.
		Moreover, it follows from the triangle inequality that 
		\begin{equation*}
			\abs{h_0}
			\leq \dist 1{u^mz} + \dist {u^m z}{u^mh^{-1}}
			< \dist 1z + \dist z{h^{-1}}
			\leq \abs{h},
		\end{equation*}
		which contradicts our definition of $h$.
	\end{proof}

	Let $\gamma \colon \intval ab \to Y$ be a geodesic of $(Y, \distV[N])$ joining $1$ to $v$.
	We identify $\gamma$ with its image in $X$.
	As a path of $X$, $\gamma$ is no more geodesic. 
	Nevertheless its length is still ${\sf Length}(\gamma) = \norm[N]{\Phi(c)}$.
	
	\begin{lemm}
	\label{res: dist p alpha}
		The path $h^{-1}\gamma$ does not intersect the ball of radius $r = \abs[Q] q - a - 3\delta$ centered at $x$.
	\end{lemm}
	
	\begin{proof}
		By construction $h\gamma^{-1}$ is contained in $h^{-1}Y$.
		Thus
		\begin{equation*}
			\dist hY \leq \dist 1{h^{-1}Y} \leq \dist 1x + \dist x{h^{-1}Y} \leq \dist x{h^{-1} \gamma} + a + 2\delta.
		\end{equation*}
		Hence it suffices to prove that $\dist hY \geq \abs[Q] q - \delta$.
		Recall that $Y$ is an embedded copy of the Cayley graph of $N$ in $G$.
		If $p$ is a point in $Y$, there exist $w \in N$ seen as a vertex of $Y$ such that $\dist pw \leq \delta$.
		Observe that $w^{-1}h$ is a pre-image of $q$.
		Since the projection $G \onto Q$ is $1$-Lipschitz, we get $\abs[Q]q  \leq \dist hw \leq \dist h{p} + \delta$.
		This inequality holds for every $p \in Y$ which completes the proof.
	\end{proof}
	
	Since $v = huh^{-1}$, the path $h^{-1} \gamma$ joins $h^{-1}$ to $uh^{-1}$ and avoid the ball centered at $x$ of radius $r$.
	It follows from \autoref{res: div} that 
	\begin{equation*}
		\norm[N]{\Phi(c)} 
		= {\sf Length}(\gamma) 
		= {\sf Length}(h^{-1} \gamma) \geq A\lambda ^r.
	\end{equation*}
	It can be rewritten $\ln \norm[N]{\Phi(c)} \geq \kappa \abs[Q]q - \ell$, where
	\begin{equation*}
		\kappa = \ln \lambda
		\quad \text{and} \quad
		\ell = (a + 3\delta)\ln \lambda - \ln A
	\end{equation*}
	do not depend on $q$.
	Note also that $\kappa$ does not depend on $c$.
	Hence the proof is complete.
\end{proof}

\begin{proof}[Proof of \autoref{res: qi embedding}]
	As in \autoref{res upper bound}, it suffices to prove the inequalities for $q' = 1$.
	The second inequality is provided by \autoref{res upper bound}.
	Let us focus on the first one.
	Since $N$ is infinite, it contains a conjugacy class $c_0$ whose elements have infinite order \cite[Chapitre~8, Corollaire~36]{Ghys:1990ki}.
	By \autoref{res: lower bound}, there exist $\kappa > 0$ and $\ell \geq 0$ such that for every $q \in Q$, 
	\begin{equation*}
		\ln \norm{\Phi(c_0)} \geq \kappa \abs q - \ell,
		\quad \text{where} \quad
		\Phi = \chi(q).
	\end{equation*}
	In particular, 
	\begin{equation*}
		\abs[\rm Lip]{\chi(q)}
		\geq \ln \left( \frac{\norm{\Phi(c_0)}}{\norm {c_0}}\right)
		\geq \kappa \abs q - \left( \ell + \ln \norm {c_0} \right).\qedhere
	\end{equation*}
\end{proof}

%
\section{Applications}
%
\label{sec: applications}

Our applications combine two ingredients: the Rips construction and the existence of groups with distorted elements.

\paragraph{Rips construction.}
In order to produce short exact sequences as studied in \autoref{sec: short exact sequence}, we use the Rips construction.

\begin{theo}[Rips \cite{Rips:1982co}]
	Let $Q$ be a finitely presented group. 
	There exists a short exact sequence
	\begin{equation*}
		1 \to N \to G \to Q \to 1
	\end{equation*}
	where $G$ is a hyperbolic group and $N$ a finitely generated non-elementary normal subgroup of $G$.
\end{theo}

Rips construction is very flexible.
In particular, one can strengthen the conclusions by adding one of the following requirements (some of them are incompatible, hence they cannot be all simultaneously satisfied):
\begin{itemize}
	\item The group $G$ is torsion-free; see Rips \cite{Rips:1982co}.
	\item The group $N$ is the quotient of a prescribed non-elementary hyperbolic group. In particular one can choose $N$ to have Kazhdan's property (T); see Belegradek-Osin \cite{Belegradek:2008cy}.
	\item The group $G$ and thus $N$ is residually finite; see Wise \cite{Wise:2003bc}.
	Actually the group $G$ is obtained by mean of small cancellation theory.
	It follows from the work of Wise \cite{Wise:2004ky, Wise:2012fo}, Agol \cite{Agol:2013wr} and Haglund-Wise \cite{Haglund:2008ie} that $G$ and thus $N$ are linear.
	\item The induced morphism $Q \to \out N$ is an isomorphism; see Bumagin-Wise \cite{Bumagin:2005fr}.
\end{itemize}

\paragraph{Distortion.}
Let $H$ be a finitely generated group.
Two length functions $L_1$ and $L_2$ on $H$ are \emph{strongly equivalent} if and only if there exists $C >0$ such that for every $h \in H$, 
\begin{equation*}
	 \frac 1C L_1(h) \leq L_2(h) \leq C L_1(h).
\end{equation*}
If $H$ is a subgroup of a finitely generated group $G$, then the word metric of $G$ restricted to $H$ provides a length function on $H$.
Ol'shanski\u\i\ proved that every length function on $G$ essentially arises in this way \cite{Olshanskii:1999ve}.
Moreover he gives a complete description of the distortion of subgroups of \emph{finitely presented} group.

\begin{theo}[Ol'shanski\u\i\ {\cite[Theorem~2]{Olshanskii:1997ga}}]
\label{res: distorted element}
	Let $H$ be a finitely generated group and $L \colon H \to \N$ a \emph{computable} length function on $H$.
	Then $H$ embeds in a finitely presented group $G$ such that $L$ is strongly equivalent to the word metric of $G$ restricted to $H$.
\end{theo}

\paragraph{Exotic growth.}

\begin{theo}
\label{res: single length function}
	Let $L \colon \Z \to \N$ be computable length function.
	There exist a finitely generated group $N$ and an outer automorphism $\Phi \in \out N$ such that for every non-trivial conjugacy class $c$ of $N$ the map
	\begin{equation*}
		\begin{array}{ccc}
			\N & \to & \R \\
			n & \to & \ln \norm{\Phi^n(c)}
		\end{array}
	\end{equation*}
	is equivalent to $L$.
\end{theo}

\begin{proof}
	According to \autoref{res: distorted element}, there exist a finitely presented group $Q$ and an element $q \in Q$ such that the map $L_q \colon \Z \to \N$ sending $n \to \abs{q^n}$ is equivalent to $L$.
	Using the Rips construction, we produce a short exact sequence
	\begin{equation*}
		1 \to N \to G \to Q \to 1
	\end{equation*}
	where $G$ is a torsion-free hyperbolic group and $N$ a non-elementary finitely generated subgroup.
	Let $\Phi$ be the image of $q$ by the morphism $\chi \colon Q \to \out N$ induced by this short exact sequence.
	Combining Propositions~\ref{res upper bound} and \ref{res: lower bound} we observe that for every non-trivial conjugacy class $c$ of $N$, the map $\N \to \R$ sending $n$ to $\ln \norm{\Phi^n(c)}$ is equivalent to $L_q$ hence to $L$.
\end{proof}

\begin{rema}
	Using one of the aforementioned variations of the Rips construction, one can build a group $N$ as in \autoref{res: single length function} which satisfies one of the following additional assumptions : $N$ has Kazhdan Property (T), $N$ is linear, $N$ is residually finite, etc.
\end{rema}

\begin{rema}
	Ol'shanski\u\i's \autoref{res: distorted element} relies on a rather heavy construction, involving among others a variation on the Higmann embedding theorem.
	In practice though, it is easy to produce \emph{explicit} examples of automorphisms with an exotic behavior.
	Consider for instance the discrete Heisenberg group whose presentation is
	\begin{equation*}
		Q = \left< a,b \mid [a, [a,b]] =1,  [b, [a,b]]=1\right>.
	\end{equation*}
	It is known that the commutator $q = [a,b]$ is distorted.
	More precisely $n \to \abs{q^n}$ behaves like $n \to \sqrt n$.
	Following Rips' construction, one can choose for $G$ the group generated by $a,b,x,y$ subject to the following relations
	\begin{align*}
		axa^{-1} & = xy^{k}x y^{k +1}x y^{k +2} \dots xy^{2k-1}, \\
		a^{-1}xa & = xy^{2k }x y^{2k+1}x y^{2k+2} \dots xy^{3k-1}, \\
		bxb^{-1} & = xy^{3k}x y^{3k+1}x y^{3k+ 2} \dots xy^{4k-1}, \\
		b^{-1}xb & = xy^{4k }x y^{4k+1}x y^{4k+2} \dots xy^{5k-1}, \\
		[a, [a,b]] & = xy^{5k }x y^{5k+1}x y^{5k+2} \dots xy^{6k-1}, \\
		[b, [a,b]] & = xy^{6k }x y^{6k+1}x y^{6k +2} \dots xy^{7k-1},
	\end{align*}
 	where $k$ is a sufficiently large exponent.
	The group $N$ given by \autoref{res: single length function} is then the (normal) subgroup of $G$ generated by $x$ and $y$.
	In particular, if $\Phi$ stands for the image of $q$ under the morphism $\chi \colon Q \to \out N$, then for every non-trivial conjugacy class $c$, the map $n \to \ln \norm{\Phi^n(c)}$ behaves like $n \to \sqrt n$.
	Other \og explicit\fg  examples of distorted elements in a finitely presented group can be found in Gromov \cite[Chapter~3]{Gromov:1993vu} or Bridson \cite{Bridson:1999aa}.
	
\end{rema}

We now prove the theorems announced in the introduction.

\begin{proof}[Proof of \autoref{res: growth types}]
	We write $L_1, \dots, L_m$ for the elements of $\mathcal L$.
	For every $i \in \intvald 1m$ we build using \autoref{res: single length function} a finitely generated group $N_i$ and an automorphism $\Phi_i \in \out{N_i}$ such that for every non-trivial conjugacy class $c$ of $N_i$ the map $\N \to \R$ sending $n$ to $\ln \norm{\Phi_i^n(c)}$ is equivalent to $L_i$.
	We now let
	\begin{equation*}
		G = N_1 \times N_2 \times \dots \times N_m.
	\end{equation*}
	We endow $G$ with the word metric relative to the generating set $S = S_1 \cup \dots \cup S_m$ where $S_i$ is a finite generating set of $N_i$.
	For each $i \in \intvald 1m$, we pick a representative $\phi_i \in \aut{N_i}$ of $\Phi_i$.
	We consider the automorphism $\phi \in \aut G$ whose restriction to the factor $N_i$ is $\phi_i$, and write $\Phi$ for the outer class of $\phi$.
	One checks easily that if $g = (g_1, \dots, g_m)$ is an element of $G$, then for every $n \in \Z$,
	\begin{equation*}
		\norm[G]{\phi^n (g)} = \norm[N_1]{\phi_1^n(g_1)} + \dots + \norm[N_m]{\phi_m^n(g_m)}.
	\end{equation*}
	Let $c$ be the conjugacy class of $g$ in $G$.
	We denote by $I$ the set which consists of all $i \in \intvald 1m$, such that $g_i$ is non-trivial.
	It follows from the previous discussion that $T_c$ is equivalent to 
	\begin{equation*}
		\sum_{i \in I} L_i
	\end{equation*}
	Conversely for every $i \in \intvald 1m$, $L_i$ is equivalent to $T_{c_i}$, where $c_i$ is the conjugacy class in $G$ of any non-trivial element in $N_i$.
\end{proof}

\begin{proof}[Proof of \autoref{res: qi}]
	According to Bumagin-Wise \cite{Bumagin:2005fr} there exists a short exact sequence 
	\begin{equation*}
		1 \to N \to G \to Q \to 1
	\end{equation*}
	where $G$ is a hyperbolic group and $N$ a non-elementary finitely generated subgroup such that the corresponding map $Q \to \out N$ is an isomorphism.
	By \autoref{res: qi embedding}, $Q \to \out N$ is a surjective quasi-isometric embedding, hence a quasi-isometry.
\end{proof}

\bigskip
\noindent
\emph{R\'emi Coulon} \\
Univ Rennes, CNRS \\
IRMAR - UMR 6625 \\
F-35000 Rennes, France\\
\texttt{remi.coulon@univ-rennes1.fr} \\
\texttt{http://rcoulon.perso.math.cnrs.fr}

\todos
\end{document}